\documentclass[11pt,twoside]{article}

\usepackage{mathrsfs,amsfonts,amsmath,amssymb}
\usepackage{latexsym}
\usepackage{comment} 
\newtheorem{satz}{Theorem}
\newtheorem{proposition}[satz]{Proposition}
\newtheorem{theorem}[satz]{Theorem}
\newtheorem{lemma}[satz]{Lemma}
\newtheorem{definition}[satz]{Definition}
\newtheorem{corollary}[satz]{Corollary}
\newtheorem{remark}[satz]{Remark}

\def\Z{\mathbb {Z}}
\def\F{\mathbb {F}}
\def\E{\mathsf{E}}

\def\a{\alpha}

\def\d{\delta}
\def\o{\omega}
\def\({\big (}
\def\){\big )}

\def\G{\Gamma}

\def\le{\leqslant}
\def\ge{\geqslant}
\def\_phi{\varphi}
\def\eps{\varepsilon}

\def\Gr{{\mathbf G}}
\def\FF{\widehat}

\def\t{\tilde}

\def\la{\lambda}

\def\SL{{\rm SL}}

\def\cov{{\rm cov}}

\newcommand{\bp}{\bigskip}

\author{I.D. Shkredov}
\title{On some multiplicative properties of large difference sets 
}
\date{}
\begin{document}
	\maketitle


\begin{center}
	Annotation.
\end{center}

{\it \small
    In our paper we study multiplicative properties of difference sets $A-A$ for large sets $A \subseteq \Z/q\Z$ in the case of  composite $q$. 
    We  obtain a quantitative version of a result of A. Fish about the structure of the product sets $(A-A)(A-A)$. 
    Also, we show that the multiplicative covering number of any difference set is always small. 
}
\\

\section{Introduction}

The landmark question about solvability of equations of the form $f(x_1,\dots,x_n) = 0$, where $f\in \Z[x_1,\dots,x_n]$ and the variables $x_j \in X_j$ belong to some ``large'' but unspecified sets $X_j$ of the prime field $\F_q$ was firstly posed, probably, in  \cite{sarkozy2005sums}.  
Interesting in its own right the problem has a clear  connection  with the sum--product phenomenon \cite{TV} due to the fact that as a rule the polynomial $f$ includes both the addition and the multiplication. 
This theme becomes rather popular last years, e.g., see  
\cite{gyarmati2008equationsI}---\cite{sarkozy2005sums},
\cite{sh_mono}, \cite{shparlinski2008solvability} and many other papers.

The question about a partial resolution of some specific equations  $f(x_1,\dots,x_n)=0$ in large subrings of rings $\Z_q := \Z/(q \Z)$ for composite $q$  was firstly considered by A. Fish in \cite{fish2018product} (nevertheless, let us remark that a similar problem was formulated in \cite[Problem 5]{gyarmati2008equationsII}). In particular, in \cite[Corollary 1.2]{fish2018product} A. Fish considered the polynomial $f(x_1,x_2,x_3,x_4) = (x_1-x_2)(x_3-x_4)$ and proved  the following result.

\begin{theorem}
    Let $q$ be a 
    positive integer, 
    $A, B\subset \Z_q$ be sets, $|A| = \a q$, $|B|=\beta q$, and suppose that $\a \ge \beta$. 
    Then there is $d|q$ 
    with  
\begin{equation}\label{f:Fish_intr}
    d \le F(\beta) 
\end{equation}
    and such that 
\begin{equation}\label{f:Fish_2_intr}
    d \cdot \Z_q \subseteq (A-A)(B-B) \,.
\end{equation}
\label{t:Fish_intr}
\end{theorem}

Here we use the following standard notation \cite{TV}, namely,
given two sets $A,B\subset \Z_q$, define  
the {\it sumset} 
of $A$ and $B$ as 
$$A+B:=\{a+b ~:~ a\in{A},\,b\in{B}\}\,.$$
In a similar way we define the {\it difference sets} $A-A$, the {\it higher sumsets}, e.g., $2A-A$ is $A+A-A$ and, further, the {\it products sets} 
$$AB:=\{ab ~:~ a\in{A},\,b\in{B}\} $$
the {\it higher product sets} and so on. 
Finally, if $A \subseteq \Z_q$ and $\la \in \Z_q$, then we write 
\[
    \la \cdot A = \{ \la a ~:~ a\in A\} \,.
\]

It is easy to see that in this generality one cannot have $\Z_q = (A-A)(B-B)$ in inclusion \eqref{f:Fish_2_intr} for all sets $A,B$ and thus we indeed need this additional (but small)  divisor $d$. 
In contrary, for prime $q$ the divisor $d$ can be omitted and the questions of this type were studied in \cite{hart2007sum}
and 
\cite{shparlinski2008solvability}.
In paper \cite{fish2018product} the dependence on    $F(\beta)$ was triple exponential on $\beta^{-1}$. Using a series of other methods, we improve and generalize the last result in several directions. 
The signs $\ll$ and $\gg$ below  are the usual Vinogradov symbols.

\begin{theorem}
    Let $q$ be a 
    positive integer, 
    $A, B\subset \Z_q$ be sets, $|A| = \a q$, $|B|=\beta q$, and suppose that $\a \ge \beta$. 
    Then there is $d|q$ 
    with  
\begin{equation}\label{f:Fish_new_intr}
    d \ll  \exp( C\beta^{-4}) \,,
\end{equation}
    where $C>0$ is an absolute constant and such that 
\begin{equation}\label{f:Fish_new_2_intr}
    d \cdot \Z_q \subseteq (A-A)(B-B) \,.
\end{equation}
\label{t:Fish_new_intr}
\end{theorem}

In \cite{fish2018product} the author posed a series of questions in  much more general form,  
as well as 
for other polynomials $f(x_1,\dots,x_n)$.
Using different approaches we partially resolve some of them, see Sections \ref{sec:general_case} and \ref{sec:group_Fish}. 
In particular, we have deal with the equation  
\[
    (a_1-b_1) (a_2 - b_2) \equiv \la \pmod q \,, \quad \quad (a_1,a_2) \in \mathcal{A},\, (b_1,b_2) \in \mathcal{B} 
    \,,
\]
and 
\[
    (a_1-b_1)^2 -  (a_2 - b_2)^2 \equiv \la \pmod q \,, \quad \quad (a_1,a_2) \in \mathcal{A},\, (b_1,b_2) \in \mathcal{B} 
\]
for rather general two--dimensional sets $\mathcal{A}, \mathcal{B} \subseteq \Z_q \times \Z_q$ and composite numbers $q$ with some restrictions on its prime divisors, see Theorems \ref{t:Fish-Vinh}, \ref{t:Fish-Vinh+} 
and Theorem \ref{t:Fish_action} below. 
As an example, we formulate a part of Theorem \ref{t:Fish-Vinh}.

\begin{theorem}
    Let $q$ be a squarefree  number, $\mathcal{A}, \mathcal{B} \subseteq \Z_q^2$ be sets, 
    $|\mathcal{A}| = \a q^2$, $|\mathcal{B}|=\beta q^2$, and suppose that $\a \ge \beta$. 
    Then 
\begin{equation}\label{f:Fish-Vinh1_intr}
    d  \cdot \Z^*_q \subseteq 
    \{ (a_1-b_1) (a_2 - b_2) ~:~ (a_1,a_2) \in \mathcal{A},\, (b_1,b_2) \in \mathcal{B} \} 
\end{equation} 
    with 
\[
    d \ll \exp (O(\o(q) \log \o(q) - \log \beta)    \,.
\]
    In particular, for $\mathcal{A} = \mathcal{B}$ one has with the same $d$ that 
\begin{equation}\label{f:Fish-Vinh1.5_intr}
    d  \cdot \Z_q \subseteq 
    \{ (a_1-b_1) (a_2 - b_2) ~:~ (a_1,a_2) \in \mathcal{A},\, (b_1,b_2) \in \mathcal{B} \} 
    \,.
\end{equation} 
\end{theorem} 

Our 
another 
result 
may be interesting in itself
(even in the case of prime $q$) due to it gives a new necessary condition for a set to be a difference set 
but moreover, in addition, it 
 yields  another proof of Theorem \ref{t:Fish_intr} (see Theorems \ref{t:covm_A-A}, \ref{t:Fish_p1} below). 

\begin{theorem}
       Let $q$ be a positive integer, $A \subseteq \Z_q$ be a set, $|A| = \a q$.
    Suppose that the least prime factor of $q$ 
    greater than 
    $2 \a^{-1}+3$. 
    Then there is $X \subseteq \Z_q$ such that 
\begin{equation*}\label{f:covm_A-A_intr}  
    |X| \le \frac{1}{\a} + 1 \,,
\end{equation*}
    and 
\[
    X(A-A) = \Z_q \,. 
\]
\label{t:covm_A-A_intr}   
\end{theorem}

The equation $X(A-A) = \Z_q$ for a set $X \subseteq \Z_q$ induces a coloring of $\Z_q$ via suitable subsets of our difference set $A-A$. 
Hence Theorem \ref{t:covm_A-A_intr}  gives us a new connection between coloring problems and difference sets.
Finally, our result and the Ruzsa covering lemma (see inclusion \eqref{f:Ruzsa_cov} below) show that for any set $A\subseteq \Z_q$, $|A| \gg q$, where $q$ is  a prime number, say,  the set  $A-A$ is  a syndetic set in both multiplicative and additive ways.

Let us say a few words about the notation.
Having a positive integer $q$ we denote by $\omega(q)$ the total number of prime divisors of $q$ and by $\tau(q)$ the number of all divisors. Let $\_phi(q)$ be the Euler function. 
We use the same capital letter to denote a set $A\subseteq \Z_q$ and   its characteristic function $A: \Z_q \to \{0,1 \}$. 
If $\mathcal{R}$ is a ring, then we write $\mathcal{R}^*$ for the group of all inverse elements of $\mathcal{R}$. 
Let $e_q(x) = e^{2\pi i x/q}$ and let us denote by $[n]$ the set $\{1,2,\dots, n\}$.
All logarithms are to base $2$.

The author thanks Alexander Fish for very useful remarks (especially Remark \ref{r:Fish} below) and comments. 

\section{An effective version of Fish's theorem}
\label{sec:Fish_improvement} 

Having a positive integer $n$ and a set $A \subseteq \Z_q \times \dots \times \Z_q = \Z^n_q$ (or just $A\subset \Z^n$), as well as a divisor $q_*|q$,  we write 
$$
    \pi_{q_*} (A) = \{ (a_1 \pmod{q_*},~ \dots ~,~ a_n \pmod{q_*}) ~:~ (a_1,\dots, a_n) \in A \} \subseteq \Z^n_{q_*} \,.
$$
We need a simple regularization result similar to \cite[Lemma 2.1]{bourgain2008sum}.

\begin{lemma}
    Let $\d,\eps \in (0,1)$, $M\ge 2$ be real numbers, $n$ be a positive integer, and $A\subset \Z^n_q$ be a set, $|A|=\d q^n$. 
    Then there is $q_*|q$, and a set $A_* \subseteq A$, 
    $|\pi_{q/q_*} (A_*)| = 1$
    such that $q_* = \frac{q}{q_1 \dots q_s}$, $M\le q_j \le \d^{-\eps^{-1}}$,
    $s$ is the least number with $\d M^{\eps s} > 1$ 
    and 
    for all $\tilde{q}|q_*$, $\tilde{q} \ge M$ one has 
\begin{equation}\label{f:regularization_Zq}
    \max_{\xi \in \Z^n_{\tilde{q}}} |A_* \cap \pi^{-1}_{\tilde{q}} (\xi)| \le \frac{|A_*|}{\tilde{q}^{1-\eps}} \,.
\end{equation}
\end{lemma}
\label{l:regularization_Zq}
\begin{proof}
    Suppose not. Then for a certain $\xi \in \Z^n_{q}$ and $q_1|q$, $q_1 \ge M$ we find $A' := A \cap \pi^{-1}_{q_1} (\xi)$ with $|A'|\ge \frac{|A|}{q^{1-\eps}_1}$.
    Clearly, $|\pi_{q_1} (A')| = 1$ and the density of $A'$ in the appropriate shift of $\Z^n_{q/q_1}$ is at least $\d q^\eps_1 \ge \d M^{\eps}$. 
    Hence applying the same procedure to the set $A'$ 
    and to the new module $q/q_1$, 
    we see that our algorithm must stop after at most $s$ steps. 
    Notice that condition \eqref{f:regularization_Zq} holds automatically if $\tilde{q} \ge \d^{-\eps^{-1}}$ and hence 
    at the final step of our procedure we find a set $A_* \subset A$, $|\pi_{q/q_*} (A_*)| = 1$, having all required properties.
    This completes the proof. 
$\hfill\Box$
\end{proof}

\bigskip 

Now we are ready to obtain the main result of this section, which implies Theorem \ref{t:Fish_new_intr} from the introduction. 
Our proof uses the Fourier analysis  (its standard facts can be found in \cite{TV}, say) and 
classical 
estimates for the Kloosterman sums. 
Having a group $\Gr$,  we define for any function $f:\Gr \to \mathbb{C}$ and a representation  $\rho \in \FF{\Gr}$ 
the Fourier transform of $f$ at $\rho$ by the formula 
\begin{equation}\label{f:Fourier_representations}
\FF{f} (\rho) = \sum_{g\in \Gr} f(g) \rho (g) \,.
\end{equation}

\begin{theorem}
    Let $q$ be a 
    positive integer, 
    $A, B\subset \Z_q$ be sets, $|A| = \a q$, $|B|=\beta q$, and suppose that $\a \ge \beta$. 
    Then there is $d|q$ 
    with  
\begin{equation}\label{f:Fish_new}
    d \ll  \exp( C\beta^{-4}) \,,
\end{equation}
    where $C>0$ is an absolute constant and such that 
\begin{equation}\label{f:Fish_new_2}
    d \cdot \Z_q \subseteq (A-A)(B-B) \,.
\end{equation}
    In addition 
\begin{equation}\label{f:Fish_new_3}
    d \ll \beta^{-O(\o(q))} \,.
\end{equation}
\label{t:Fish_new}
\end{theorem}
\begin{proof}
   Let $q=p_1^{\rho_1} \dots p_t^{\rho_t}$, where $p_j$ are different primes, $p_1<\dots<p_t$.
   Also, let $M\ge 2$, $\eps \in (0,1)$ be parameters,  which we will choose later. 
   First of all, we remove all divisors less than $M$ from $q$. 
   More precisely, for any $p_j$, $j\in [t]$ let  $\gamma_j \le \rho_j$ be the maximal nonnegative integer such that $p^{\gamma_j}_j \le M$.
   Clearly, $\gamma_1 \ge \gamma_2 \ge \dots \gamma_t \ge 0$ and let $t_0 \le t$ be the maximal $j$ with 
   $\gamma_j \neq 0$. 
   Thus $t_0 \le \pi (M)$. 
   Now we define 
   \begin{equation}\label{f:Q_1_def}
   Q_1 := \prod_{j=1}^{t_0} p_j^{\gamma_j} \le M^{t_0} \le \min\{ M^{\pi (M)}, M^{\omega(q)} \} 
   \end{equation}
   and take $A_1 \subseteq A$ such that a shift of $A_1$ belongs to  $\Z_{q/Q_1}$ and has density at least $\a$.
   In particular, $|\pi_{Q_1} (A_1)|=1$ and 
   of course such a shift exists by the Dirichlet principle. 
   Similarly, we can do the same with the set $B$ so as not to lose the density.
   Secondly, we apply Lemma \ref{l:regularization_Zq} 
   with $n=1$, $A=A_1$ to regularize the set $A_1$ and find a set $A_* \subseteq A_1$ and a module $q_*$ 
   that 
   satisfies 
   \eqref{f:regularization_Zq} and all other restrictions. 
   Again,  using the Dirichlet principle, we take $B_* \subseteq B$ such that the density of $B$ does not decrease. 
    Let $\la \in \Z_{q_*}$ be an arbitrary number and we first suppose that $\la \in \Z^*_{q_*}$. 
    To prove $\la \in (A-A)(B-B)$ it is 
    enough 
    to show that $\la \in (A_*-A_*)(B_*-B_*)$ or, equivalently, in terms of the Fourier transform
    it suffices 
    obtain 
    the inequality 
\begin{equation}\label{tmp:01.12_1}
    \frac{|A_*|^2 |B_*|^2}{q_*} > \frac{1}{q_*} \sum_{r\neq 0} |\FF{B}_* (r)|^2 \cdot \sum_{a_1,a_2\in A_*} e_{q_*} \left(\frac{\la r}{a_1-a_2} \right) 
     := \sigma \,.
\end{equation}
    Now clearly,  
\begin{equation}\label{tmp:eta_K-}
    \sigma \le \frac{1}{q_*} \sum_{q_2|q_*,\, q_2>1}\, \sum_{z\in \Z^*_{q_2}} |\FF{B}_* (zq_* q^{-1}_2)|^2 
    \left| \sum_{a_1,a_2\in A_*} e_{q_2} \left(\frac{\la z}{a_1-a_2} \right) \right| \,.
\end{equation}
    In terms of the Kloosterman sums 
\[
    K_q (\la,r) := \sum_{x\in \Z^*_{q}} e_{q} \left( \frac{\la}{x} + rx\right)  
\]
    and the density function
\begin{equation}\label{def:eta}
    \eta_{q_2} (\xi) := |\{ a\in A_* ~:~ a \equiv \xi \pmod {q_2} \}|
\end{equation}
    one has (recall that $\la \in \Z^*_{q_*}$ and $z \in \Z^*_{q_2}$)
\begin{equation}\label{tmp:eta_K}
    \sum_{a_1,a_2\in A_*} e_{q_2} \left(\frac{\la z}{a_1-a_2} \right)
    = 
    \sum_{\xi_1,\xi_2 \in  \Z_{q_2}} \eta_{q_2}(\xi_1) \eta_{q_2} (\xi_2) e_{q_2} \left(\frac{\la z}{\xi_1-\xi_2} \right)
    =
    q^{-1}_2 \sum_{\xi \in \Z_{q_2}} |\FF{\eta}_{q_2} (\xi)|^2 K_{q_2} (\la z, \xi)
\end{equation} 
\[
    \le 2 \sqrt{q_2} \tau(q_2) \| \eta_{q_2} \|_2^2 
    \,.
\]
    In the last line we have applied the well--known bound for the Kloosterman sum and the Parseval identity.
    Now to estimate $\| \eta_{q_2}\|_2^2$ we use the regularity property of $A_*$ and 
    derive  
\begin{equation}\label{f:eta_L2}
    \| \eta_{q_2}\|_2^2 \le \| \eta_{q_2}\|_\infty \| \eta_{q_2}\|_1
    \le 
    \frac{|A_*|^2}{q^{1-\eps}_2} \,.
\end{equation} 
    Further let us obtain a lower bound for divisors $q_2$. 
    Since $|\pi_{Q_1} (A_1)|=1$, it follows that for all $q_1|Q_1$, we have 
\[
    \frac{1}{q_1} \sum_{\xi \in \Z_{q_1}} |\FF{\eta}_{q_1} (\xi)|^2 K_{q_1} (\la z, \xi)
    =
    \frac{|A_*|^2}{q_1} \sum_{\xi \in \Z_{q_1}} K_{q_1} (\la z, \xi) = 0 \,.
\]
    Thus one can see that summations in \eqref{tmp:eta_K-} 
    is taken over $q_2 \ge M$.    
    Choosing  $\eps=1/4$, say, and 
    using the last fact,
    we get in view of the Parseval identity that
\[
    \sigma \ll M^{-1/4} \frac{|A_*|^2}{q_*} \sum_{q_2|q_*,\, q_2>1}\, \sum_{z\in \Z^*_{q_2}} |\FF{B}_* (zq_* q^{-1}_2)|^2 
    \le 
    M^{-1/4} |A_*|^2 |B_*| \,.
\]
    Returning to 
    \eqref{tmp:01.12_1}, we obtain a contradiction provided  $|B_*| \gg q_* M^{-1/4}$.
    In other words, we have for a certain $s\ge 0$,  $\a M^{s/4} > 1$ that 
\[
    |B| M^{s/4} \ll q M^{-1/4} 
\]
    and this implies $M \ll \beta^{-4}$. 
    Thus in view of our restriction to the divisors of $q_*$, the condition $\a \ge \beta$, the first bound for $Q_1$ from \eqref{f:Q_1_def},  and the  bound for $s$,  which follows from Lemma \ref{l:regularization_Zq}, we get 
\[
    d \ll M^{\pi (M)} \exp (O(\log^2 (1/\a) / \log (1/\beta)) \ll \exp(O(\beta^{-4})) 
\]
    as required.

    Now let $\la \in \Z_{q_*}$ be an arbitrary element. 
    Write $\la = q' \la'$, where $q' |q$ and $\la'\in \Z^*_{q_*/q'}$.
    Using the Dirichlet principle, choose  a subset of $B' \subseteq B_*$ of density at least $\beta$ such that all elements of a shift of $B'$ are divisible by $q'$. Then our 
    inclusion 
    can be rewritten as $\la' \in (A_*-A_*)(B'-B')$ modulo $q_*/q'$ and we can apply the arguments above replacing module $q_*$ to $q_*/q'$.

    To obtain \eqref{f:Fish_new_3} we use the second bound for $Q_1$ from \eqref{f:Q_1_def} and derive as above 
\[
    d \ll M^{\o (q)} \exp (O(\log (1/\beta))) 
    \ll
        \exp (O(\o (q) \log (1/\beta))) \,.
\]
    This completes the proof. 
$\hfill\Box$
\end{proof}

\bp 

As one can see from the proof of Theorem \ref{t:Fish_new} that the constant four in \eqref{f:Fish_new} can be decreased to $2+o(1)$ but we 
leave 
such calculations to the interested reader.

\section{On the general case}
\label{sec:general_case}

In \cite[Problem 2]{fish2018product} A. Fish considered a more general two--dimensional case (actually, in his paper he had to deal with even more general dynamical setting) and formulated  the following problem. 


{\bf Problem.}
{\it Let $q$ be a positive number and $\mathcal{A}, \mathcal{B} \subseteq \Z_q^2$ be sets, $|\mathcal{A}| = \a q^2$, $|\mathcal{B}|=\beta q^2$, and suppose that $\a \ge \beta$. 
Prove that in the case $\mathcal{A} = \mathcal{B}$ for a certain function $F$ there is $d|q$ such that $d\le F(\beta)$ and 
\begin{equation}\label{f:F=(A-A)(B-B)}
    d \cdot \Z_q \subseteq  \{ (a_1-b_1) (a_2 - b_2) ~:~ (a_1,a_2) \in \mathcal{A},\, (b_1,b_2) \in \mathcal{B} \} \,, 
\end{equation}
    provided $\beta$ is sufficiently large.
}

\bp

In this section we study the number $N_{\mathcal{A}, \mathcal{B}} (\la)$ of the solutions 
 to the equation 
\begin{equation}\label{f:N=(A-A)(B-B)}
    (a_1-b_1) (a_2 - b_2) \equiv \la \pmod q \,, \quad \quad (a_1,a_2) \in \mathcal{A},\, (b_1,b_2) \in \mathcal{B} \,.
\end{equation}
and give a partial answer to the problem above.
We consider the squarefree case for simplicity and
emphasis 
one more time that our sets $\mathcal{A}$, $\mathcal{B}$ are arbitrary (in the case of Cartesian products and squarefree $q$ one can apply other methods, see \cite{pach2013ramsey}). 
Also, in the case of prime $q$ we obtain a result of Vinh--type \cite{vinh2011szemeredi}, see asymptotic formula \eqref{f:Fish-Vinh2} below.

\begin{theorem}
    Let $q$ be a squarefree  number, $\mathcal{A}, \mathcal{B} \subseteq \Z_q^2$ be sets, 
    $|\mathcal{A}| = \a q^2$, $|\mathcal{B}|=\beta q^2$, and suppose that $\a \ge \beta$. 
    Then 
\begin{equation}\label{f:Fish-Vinh1}
    d  \cdot \Z^*_q \subseteq 
    \{ (a_1-b_1) (a_2 - b_2) ~:~ (a_1,a_2) \in \mathcal{A},\, (b_1,b_2) \in \mathcal{B} \} 
\end{equation} 
    with 
\[
    d \ll \exp (O(\o(q) \log \o(q) - \log \beta)    \,.
\]
    In particular, for $\mathcal{A} = \mathcal{B}$ one has with the same $d$ that 
\begin{equation}\label{f:Fish-Vinh1.5}
    d  \cdot \Z_q \subseteq 
    \{ (a_1-b_1) (a_2 - b_2) ~:~ (a_1,a_2) \in \mathcal{A},\, (b_1,b_2) \in \mathcal{B} \} 
    \,.
\end{equation} 
    In the case when $q$ is a prime number, we have 
\begin{equation}\label{f:Fish-Vinh2}
    \left| N_{\mathcal{A}, \mathcal{B}} (\la) - \frac{|\mathcal{A}| |\mathcal{B}|}{q} \right| 
        <
        4 q^{7/8} \sqrt{|\mathcal{A}||\mathcal{B}|} \,.
\end{equation} 
    In particular, equality \eqref{f:F=(A-A)(B-B)} holds for $|\mathcal{A}| |\mathcal{B}| \ge 16 q^{15/4}$ and $d=1$.
\label{t:Fish-Vinh}
\end{theorem}
\begin{proof}
    We start with \eqref{f:Fish-Vinh1}.
    The proof follows the arguments of the proof of Theorem \ref{t:Fish_new} and thus we use the notation from this result. 
    In particular, writing $q=p_1^{\rho_1} \dots p_t^{\rho_t}$, $t=\omega(q)$ with $\rho_j=1$, $j\in [t]$ we define $Q_1 = \prod_{j=1}^s p_j$ such that \eqref{f:Q_1_def} 
    holds and further we take  $\la \in \Z^*_q$.
    The only difference is that one should use Lemma \ref{l:regularization_Zq} with $n=2$ to regularize the two--dimensional  set $\mathcal{A}$ and let $\varepsilon =1/4$.
    For a moment we  assume that $M \ge 100 t^2$, say, and we will choose the parameter $M$  later. 
    Finally, with some abuse of the notation we do not use new letters $\mathcal{A}_*, \mathcal{B}_*$, $q_*$ below but the old ones $\mathcal{A}, \mathcal{B}$ and $q$ (in other words, one can think that $\mathcal{A}$ is a regularized set already).
    Also, we utilize the fact that $\Z_q = \Z_{p^{\rho_1}_1} \times \dots \times \Z_{p^{\rho_t}_t} = \Z_{p^{}_1} \times \dots \times \Z_{p^{}_t}$ thanks  the Chinese  remainder theorem. 

    Now for $a = (a_1,a_2)$ and $b=(b_1,b_2)$ let us write $I(a,b) = 1$ if the pair $a,b$ satisfies \eqref{f:N=(A-A)(B-B)} and $I(a,b) = 0$ otherwise.
    Then clearly, 
\begin{equation}\label{tmp:22.12_-1}
    N_{\mathcal{A}, \mathcal{B}} (\la) = \sum_{a\in \mathcal{A}, b\in \mathcal{B}} I(a,b) \,.
\end{equation}
    Without loosing of the generality we assume that $\la=1$. 
    Obviously, $I(a,b) = I(b,a)$ and we can rewrite the matrix $I(a,b)$ as 
    $I(a,b) = \sum_{j=1}^{q^2} \mu_j u_j (a) \overline{u}_j (b)$, where $\mu_j$ are eigenvalues and $u_j (x)$ are correspondent normalized eigenfunctions of $I$. 
    One can easily check that $u_1 (x) = q^{-1} (1,\dots, 1)$, $\|u_1\|_2 = 1$ and $\mu_1 = |\Z^*_q| = \_phi(q)$.
    Writing  $I'(a,b) = I(a,b) - \mu_1 u_1 (a) \overline{u}_1 (b)$, 
    we obtain 
\begin{equation}\label{tmp:22.12_-1'}
    N_{\mathcal{A}, \mathcal{B}} (\la) -
    \frac{|\mathcal{A}| |\mathcal{B}| \_phi (q)}{q^2}
    =
    \sum_{a\in \mathcal{A}, b\in \mathcal{B}} I'(a,b) := N'_{\mathcal{A}, \mathcal{B}} (\la) \,.
\end{equation}
    By the Cauchy--Schwarz inequality, we get
\begin{equation}\label{tmp:28.12_1}
    N'_{\mathcal{A}, \mathcal{B}} (\la)^2 \le |\mathcal{B}| 
    \sum_{a,a'\in  \mathcal{A}} \sum_b I'(a,b) I'(a',b) 
    = 
    |\mathcal{B}| 
    \sum_{a,a'\in  \mathcal{A}} (I')^2 (a,a') := 
    |\mathcal{B}| \cdot \sigma \,.
\end{equation}
    Here $(I')^2$ is the second power of the matrix $I'$. 
    Similarly, 
    $I^2 (a,a') = \sum_{b} I(a,b) I(a',b)$ and the last quantity coincides with the number of the solutions to the equation 
    \begin{equation}\label{tmp:22.12_1}
        a_2 - a'_2 = \frac{a'_1-a_1}{(a_1+x)(a'_1+x)} \,,
    \end{equation}
    where $b=(x,y)$, $a=(a_1,a_2)$ and $a'=(a'_1,a'_2)$. 
    Assume that $a\neq a'$ and rewrite our  equation \eqref{tmp:22.12_1} as 
\begin{equation}\label{eq:quadratic}
    x^2 + (a_1+a'_1)x + a_1 a'_1 + \frac{a_1-a'_1}{a_2-a'_2} = 0 \,,
\end{equation}
    and its discriminant  is $D'(a,a') := (a_1-a'_1) (a_2-a'_2)^{-1} [(a_1-a'_1) (a_2-a'_2) - 4]$. 
    Notice that if $a=a'$, then we have $\_phi (q)$ solutions to equation \eqref{tmp:22.12_1}. 
    By $\chi_p$ denote  the Legendre symbol modulo a prime $p$ and let $\chi_0$ be the main character (modulo $p$). 
    We have the identity  $\chi_p (x^{-1}) = \chi_p (x)$, $x\in \Z_p^*$ and hence
    $\chi_p (D'(a,a')) = \chi_p ( (a_1-a'_1) (a_2-a'_2)^{} [(a_1-a'_1) (a_2-a'_2) - 4] := \chi_p (D(a,a'))$. 
    In view of the Chinese remainder theorem, 
    and our choice of the regularized set $\mathcal{A}$ 
    one has 
\begin{equation}\label{tmp:16.01_1}
    I^2 (a,a') = \prod_{j=s+1}^t \left( \chi_{p_j} (D(a,a')) + \chi_0 (D(a,a')) +  (p_j-1) \d_{p_j} (a_1-a'_1, a_2-a'_2) \right) 
\end{equation}
\begin{equation}\label{tmp:16.01_2}
= \mathcal{E} (a,a') + \prod_{j=s+1}^t (\chi_0 (D(a,a'))+ (p_j-1) \d_{p_j} (a_1-a'_1, a_2-a'_2)) 
= \mathcal{E} (a,a') + \mathcal{E}' (a,a') \,,
\end{equation}
    where for a positive integer $m$ we have put $\d_m (z,w) = 1$ if $z \equiv w \equiv 0 \pmod m$, and zero otherwise. 
    Equivalently, writing $T$ for the segment $[s+1,t]$,  
    one has 
\[
    \mathcal{E} (a,a') = \sum_{\emptyset \neq S \subseteq T}\,  \prod_{j\notin S} (\chi_0 (D(a,a')) + (p_j-1) \d_{p_j} (a_1-a'_1, a_2-a'_2)) \cdot \prod_{j\in S} \chi_{p_j} ( D(a,a')) 
\]
\[
=
    \sum_{\emptyset \neq S \subseteq T}\,  \prod_{j\notin S} w_{p_j} (a,a')
    \cdot \prod_{j\in S} \chi_{p_j} ( D(a,a')) 
    \,.
\]
    Notice that  $\mathcal{E} (a,a) = 0$.
    From \eqref{tmp:16.01_1}, \eqref{tmp:16.01_2}, it follows that 
    $\mathcal{E} u_1 = 0$.
    Indeed, we know that $I^2 u_1 = \mu_1^2 u_1 = \_phi^2 (q) u_1$ and 
\begin{equation}\label{tmp:16.01_3}
    \sum_a \prod_{j=s+1}^t (\chi_0 (D(a,a'))+ (p_j-1) \d_{p_j} (a_1-a'_1, a_2-a'_2)) 
\end{equation}
\begin{equation}\label{tmp:16.01_3.5}
    =
    \prod_{j=s+1}^t \left( \sum_{z,w\in \Z_{p_j}} \chi_0((zw)^2 - 4zw) + p_j-1 \right) 
\end{equation}
\begin{equation}\label{tmp:16.01_4}
    =
    \prod_{j=s+1}^t \left( (p_j-1) \sum_{z\in \Z_{p_j}} \chi_0(z^2 - 4z) + p_j-1 \right) 
    =
    \prod_{j=s+1}^t (p_j-1)^2 = \_phi^2 (q) \,.
\end{equation}
    Hence 
    in very deed 
    $\mathcal{E} u_1 = 0$ and thus    
\begin{equation}\label{f:sigma_E,E'}
\sigma = \langle (I')^2 \mathcal{A}, \mathcal{A} \rangle = 
\langle (I')^2 f_{\mathcal{A}}, f_{\mathcal{A}} \rangle
=
\langle I^2 f_{\mathcal{A}}, f_{\mathcal{A}} \rangle
=
\langle \mathcal{E} \mathcal{A}, \mathcal{A} \rangle + \langle \mathcal{E}' f_{\mathcal{A}}, f_{\mathcal{A}} \rangle\,,
\end{equation}
    where $f_{\mathcal{A}} (a) = \mathcal{A} (a) - \langle \mathcal{A}, u_1 \rangle u_1 (a)$, 
    $\sum_a f_{\mathcal{A}} (a) = 0$.
    Let us estimate the term $r:=\langle \mathcal{E}' f_{\mathcal{A}}, f_{\mathcal{A}} \rangle$ rather roughly. 
    Since the function $f_{\mathcal{A}}$ is orthogonal to $u_1$ and $\| f_{\mathcal{A}}\|_\infty \le 1$, 
    it follows that 
\[
    |r|= \left| \sum_{a,a'} f_{\mathcal{A}} (a) f_{\mathcal{A}} (a') \prod_{j=s+1}^t \left(1- \d_{p_j} (D(a,a')) +(p_j-1)  \d_{p_j} (a_1-a'_1, a_2-a'_2) \right) \right|
\]
\[
    \le
    \sum_{\emptyset \neq S\subseteq T} \left| \sum_{a,a'} f_{\mathcal{A}} (a) f_{\mathcal{A}} (a') \prod_{j\in S} (- \d_{p_j} (D(a,a')) +(p_j-1)  \d_{p_j} (a_1-a'_1, a_2-a'_2) ) \right|
\]
\[
    \le
    2|\mathcal{A}| q^2 \sum_{n=1}^{t-s}\,  \sum_{S\subseteq T,\, |S|=n}\, \prod_{j\in S} \left( \frac{3}{p_j} + \frac{p_j-1}{p^2_j} \right)
    \le 
    2|\mathcal{A}| q^2 \sum_{n=1}^{t-s}\, \binom{t-s}{n} \left(\frac{4}{M} \right)^n
\]
\begin{equation}\label{f:sigma_E,E'+}
    \le 10 |\mathcal{A}| q^2 t M^{-1}
    \,.
\end{equation}

    Now returning to the definition of the operator $\mathcal{E} (a,a')$, recalling estimate \eqref{tmp:28.12_1} and using the  Cauchy--Schwarz inequality, 
    we obtain 
\[
    \sigma^2 
    \le |\mathcal{A}| \sum_{a,a' \in \mathcal{A}}  \sum_{x,y} \sum_{\emptyset \neq S_1,S_2 \subseteq T}\, \prod_{i\in S_1,\, j\in S_2} \chi_{p_i} (D((x,y),(a_1,a_2)) \chi_{p_j} (D((x,y),(a'_1,a'_2))
\]
\begin{equation}\label{tmp:22.12_3}
 \prod_{i\notin S_1,\, j\notin S_2} w_{p_i} ((x,y),(a_1,a_2)) w_{p_j} ((x,y),(a'_1,a'_2))
     \,.
\end{equation}
    The term with  $a \equiv a' \pmod q$ gives us a contribution at most $4^{t} |\mathcal{A}| q^2$ into the last sum (see \eqref{tmp:16.01_3}---\eqref{tmp:16.01_4} to estimate $\| w_{p_j}\|_1$ for $j\notin S$ and use the trivial fact that $\|\chi_p \|_{\infty} \le 1$ to bound the rest).
    Now let $a \neq a' \pmod q$ but $a \equiv a' \pmod {q_*}$ with maximal $q_*|q$.
    Thus $q_* \neq q$ and $Q_1|q_*$. 
    We can write $q_* = q_* (W) = Q_1 \prod_{j\in W} p_j$ for a certain (possibly empty) set $W \subseteq T$.
    Let us say that all primes $p$ such that $p|(q/q_*)$ (that is, $p|q$ and $p\notin W$) are {\it good}.
    In particular, for all good primes $p$ one has $p>M$.
    Now for a good prime $p$ the sum above $\sum_{x,y \mod \Z_p}\chi_p (D(x,y), (a_1,a_2))$ (or, analogously, the sum $\sum_{x,y \mod \Z_p}\chi_p (D(x,y), (a'_1,a'_2))$) is either at most   $3p^{3/2}$ by Weil, 
    or the sum over $y$ is $p$ if  
    $\frac{2}{x-a_1} + a_2 = \frac{2}{x-a'_1} + a'_2$ modulo $p$. 
    The last equation is nontrivial one by our choice of $p$ hence it has at most two solutions and thus 
    in any case 
    the sum over $x,y \mod \Z_p$ is at most $3p^{3/2} < 3p^2/\sqrt{M}$.
    Further we split the sets $S_1,S_2$ as $S_1 = S^*_1 \bigsqcup G_1$, $S_2 = S^*_2 \bigsqcup G_2$, where (possibly empty) sets $G_1, G_2$ correspond to good primes 
    and the sets $S^*_1 \subseteq W$, $S^*_2 \subseteq W$ correspond to the  divisors of $q_* (W)$. 
    Since $S_1, S_2 \neq \emptyset$,
    it follows 
    that either 
    $G_1 \bigcup G_2 \neq \emptyset$ 
    or $S^*_1,S^*_2 \neq \emptyset$. 
Using the notation as in \eqref{def:eta}, namely, 
\begin{equation}\label{def:eta2}
    \eta_{\t{q}} (\xi) := |\{ a\in \mathcal{A} ~:~ a \equiv \xi \pmod {\t{q}} \}| \,, \quad \quad \tilde{q}|q,\, \xi \in \Z^2_{\t{q}} 
    \,, 
\end{equation}
    we see that the number of pairs $a\equiv a' \pmod {\t{q}}$ is exactly $\| \eta_{\t{q}} \|_2^2$ for any $\t{q} |q$ and one can use bound \eqref{f:eta_L2} to estimate the last quantity.
    Now 
    recalling 
    inequality 
    \eqref{f:regularization_Zq} and splitting sum \eqref{tmp:22.12_3} according the case $W \neq \emptyset$ or not, we get 
\[
    \sigma^2 |\mathcal{A}|^{-1} \le 
\]
\begin{equation}\label{tmp:28.12_2}
    4^{t} |\mathcal{A}| q^2 + 
    q^2 \sum_{\emptyset \neq W \subseteq T}\, \sum_{a,a' \in \mathcal{A},\, a \equiv a' \pmod {q(W)}} 4^{|W|}
    +
    q^2 \sum_{a,a' \in \mathcal{A}} \sum_{n+m\ge 1} \binom{t-s}{n} \binom{t-s}{m} \left(\frac{3}{\sqrt{M}} \right)^{n+m}
%
\end{equation}
\[
    \le 
     4^{t} |\mathcal{A}|^{} q^2 
     +
     q^2 |\mathcal{A}|^{2} \sum_{\emptyset \neq W \subseteq T} 4^{|W|} M^{-3|W|/4}
     +
     4 q^2 |\mathcal{A}|^{2} t^{} M^{-1/2} 
     \ll 
     q^2 |\mathcal{A}|^{2} t^{} M^{-1/2} \,.
\]
    Using \eqref{tmp:22.12_-1}, \eqref{tmp:22.12_-1'},  
    \eqref{tmp:28.12_1}, \eqref{f:sigma_E,E'+}
    and the Cauchy--Schwarz inequality, we get
\begin{equation}\label{tmp:28.12_2'}
    N_{\mathcal{A}, \mathcal{B}} (\la) - \frac{|\mathcal{A}| |\mathcal{B}| \_phi (q)}{q^2} 
    \ll 
    (t^{} M^{-1/2})^{1/4} \cdot |\mathcal{A}|^{3/4} \sqrt{|\mathcal{B}| q}  
    +
    (t M^{-1})^{1/2} \cdot \sqrt{|\mathcal{A}| |\mathcal{B}|} q \,.
\end{equation} 
    We have $\_phi (q) \gg q/\log t$ and hence after some calculations we see that $N_{\mathcal{A}, \mathcal{B}} (\la)> 0$ provided 
    $M\gg t^{2} \beta^{-6}  \log^8 t$. 
    As in Theorem \ref{t:Fish_new}, 
    one has 
    $Q_1 \le M^t$ and thus 
\[
    d \ll \exp ( t \log M ) = \exp (O(t \log t - \log \beta) ) \,. 
\]

    In the case of prime $q$ the argument is even simpler because one do not need the regularization, the second term in \eqref{tmp:28.12_2} plus the quantity $r$ is negligible, see estimate \eqref{f:sigma_E,E'+}.
    Finally, let $\mathcal{A} = \mathcal{B}$ and if $\lambda \notin \Z^*_q$, then 
     write $\la = q' \la'$, where $q' |q$ and $\la'\in \Z^*_{q_*/q'}$.
    Using the Dirichlet principle, choose  a subset of $\mathcal{A}' \subseteq \mathcal{A}$ of density at least $\a$ such that 
    $|\pi_{q'} (\mathcal{A}')|=1$.
    Then  
    the required 
    inclusion \eqref{f:Fish-Vinh1.5} 
    can be rewritten as 
    \[
    \la' \in \{ (a_1-b_1) (a_2 - b_2) ~:~ (a_1,a_2),  (b_1,b_2) \in \mathcal{A}  \}
    \]
    and we can apply the arguments above replacing $q_*$ to $q_*/q'$. 
    This completes the proof. 
$\hfill\Box$
\end{proof}



\begin{remark}
    Of course, inclusion \eqref{f:Fish-Vinh1.5} does not hold for $\mathcal{A} \neq \mathcal{B}$, just take $\mathcal{A} = (d  \cdot \Z_q) \times (d  \cdot \Z_q)$ and  $\mathcal{B} = (d  \cdot \Z_q + 1) \times (d  \cdot \Z_q + 1)$ for an arbitrary $d|q$, $1< d \ll 1$.  
    Also, the author thinks that the error term in \eqref{f:Fish-Vinh2} can be improved but this weaker bound is enough for us to resolve our equation for sets of positive densities.  
\end{remark}

\begin{remark}
    The attentive reader may be alerted that we have two different main terms in \eqref{tmp:01.12_1} and in \eqref{tmp:28.12_2'}. Nevertheless, they are asymptotically the same due to the fact that in \eqref{tmp:28.12_2'} our parameter $M$ depends on growing quantity $\omega(q)$.  
\end{remark}


Similarly, we obtain an affirmative answer to \cite[Problem 1]{fish2018product} in the case of squarefree $q$.    
By $M_{\mathcal{A}, \mathcal{B}} (\la)$ denote the number of the solutions to the equation 
\begin{equation}\label{f:N=(A-A)^2-(B-B)^2}
    (a_1-b_1)^2 -  (a_2 - b_2)^2 \equiv \la \pmod q \,, \quad \quad (a_1,a_2) \in \mathcal{A},\, (b_1,b_2) \in \mathcal{B} \,.
\end{equation}

\begin{theorem} 
    Let $q$ be a squarefree  number, $\mathcal{A}, \mathcal{B} \subseteq \F_q^2$ be sets, 
    $|\mathcal{A}| = \a q^2$, $|\mathcal{B}|=\beta q^2$, and suppose that $\a \ge \beta$. 
    Then 
\begin{equation}\label{f:Fish-Vinh1+}
    d \Z^*_q \subseteq 
    \{ (a_1-b_1)^2 - (a_2 - b_2)^2 ~:~ (a_1,a_2) \in \mathcal{A},\, (b_1,b_2) \in \mathcal{B} \} 
\end{equation} 
    with 
\[
    d \ll \exp (O(\o(q) \log \o(q) - \log \beta)    \,.
\]
    In particular, for $\mathcal{A} = \mathcal{B}$ one has with the same $d$ that 
\begin{equation}\label{f:Fish-Vinh1.5+}
    d \Z_q \subseteq 
    \{ (a_1-b_1)^2 -  (a_2 - b_2)^2 ~:~ (a_1,a_2) \in \mathcal{A},\, (b_1,b_2) \in \mathcal{B} \} 
    \,.
\end{equation} 
    In the case when $q$ is a prime number one has 
\begin{equation}\label{f:Fish-Vinh2+}
    M_{\mathcal{A}, \mathcal{B}} (\la) - \frac{|\mathcal{A}| |\mathcal{B}|}{q}
        <
        4 q^{7/8} \sqrt{|\mathcal{A}||\mathcal{B}|} \,.
\end{equation} 
\label{t:Fish-Vinh+}
\end{theorem} 
\begin{proof} 
    The argument differs from the proof of Theorem \ref{t:Fish-Vinh} in some unimportant details only, so we use the notation from the former result. 
    Indeed,  for $a = (a_1,a_2)$ and $b=(b_1,b_2)$ we write $\t{I}(a,b) = 1$ if the pair $a,b$ satisfies \eqref{f:N=(A-A)^2-(B-B)^2} and $\t{I}(a,b) = 0$ otherwise.
    Calculating $\t{I}^2 (a,a')$, we arrive to the equation 
\begin{equation}\label{tmp:10.01_1}
    a_1^2 - (a'_1)^2 + 2(a'_1 - a_1) x - a_2^2 + (a'_2)^2 + 2(a_2 - a'_2) y = 0
\end{equation}
    and hence we can find $x$ via $y$ or $y$ via $x$, provided $a\neq a' \pmod q$. 
    Assuming that $a'_2 \neq a_2$, say, we 
    derive
\[
    y= \frac{(a'_1)^2-a_1^2 +a_2^2 - (a'_2)^2}{2(a_2 - a'_2)} + \frac{a_1-a'_1}{a_2 - a'_2} \cdot x = s + t x \,,
\]
    and hence substituting the last expression into 
    \eqref{f:N=(A-A)^2-(B-B)^2} 
    and 
    computing 
    the discriminant $\t{D} (a,a')$ (without loss of the generality, we  put $\la=1$), one obtains 
\[
    \t{D}(a,a') = (t (a_2-s)-a_1)^2 + (1-t)^2 (1+(a_2-s)^2 - a_1^2) 
\]
\begin{equation}\label{tmp:16.01_5} 
    =
    2t (t-1) (a_2-s)^2 - 2a_1 t(a_2-s) + (1-t)^2 (1-a_1^2) + (a_2-s)^2  + a_1^2 
    \,.
\end{equation} 
    As in the proof of Theorem \ref{t:Fish-Vinh} we consider $\mathcal{E} (a,a')$, take good primes and so on. 
    The first eigenvalue $\mu_1$ equals the number of the solutions to the equation $x^2-y^2 \equiv 1 \pmod q$, that is, $\_phi (q)$ again. 
    Also, $\t{I}^{2} (a,a) = \mu_1$ and for $a\neq a'$ the quantity $\t{I}^{2} (a,a')$ expressed exactly as in \eqref{tmp:16.01_1} (with 
    another discriminant 
    $\t{D}$, of course) and thus 
    one can 
    check 
    that $\mathcal{E} u_1$ vanishes 
    making calculations as in \eqref{tmp:16.01_3}---\eqref{tmp:16.01_4}. 
    Further as in Theorem \ref{t:Fish-Vinh}  
    we apply the standard Weil bound to estimate  the sum of characters. 
    For any good prime $p$ it gives us a nontrivial 
    bound 
    of the form $O(p^{3/2}) = O(p^2/\sqrt{M})$ and hence we obtain \eqref{f:Fish-Vinh1+}
    and 
    thus 
    \eqref{f:Fish-Vinh1.5+} by the same argument as at the end of Theorem \ref{t:Fish-Vinh} 
    (one can 
    check or see below that all obtained varieties are non--degenerated). 
    Finally, to get \eqref{f:Fish-Vinh2+} we need to estimate 
\[
    \sum_{a,a' \in \mathcal{A}}  \sum_{x,y}  \chi_{q} (\t{D}((x,y),(a_1,a_2))) \chi_{q} (\t{D}((x,y),(a'_1,a'_2)))
\]
    and by the Weil estimate it is at most $20  q^{3/2}$, say, excluding the case
    $\t{D}((x,y),(a_1,a_2))$ is proportional to $\t{D}((x,y),(a'_1,a'_2))$.  
    In particular, it means that the coefficients of these polynomials are proportional ones and 
    using \eqref{tmp:16.01_5}  and  comparing the coefficients before the highest degrees in $x$, say, we get  
    $\frac{a_2-2a_1-y}{(y-a_2)^4} = \frac{a'_2-2a'_1-y}{(y-a'_2)^4}$.
    Again, 
    thanks to $a \neq a'$ we see that 
    this equation is nontrivial one and hence it has at most four solutions.
    It follows  that 
    our sum is at most $4q$ in this case. 
    Thus as in \eqref{tmp:28.12_2}, \eqref{tmp:28.12_2'} we have 
\[
    M_{\mathcal{A}, \mathcal{B}} (\la) - \frac{|\mathcal{A}| |\mathcal{B}| \_phi (q)}{q^2} 
    \le
    3
    (q^{3/2} |\mathcal{A}| )^  {1/4} \sqrt{|\mathcal{A}| |\mathcal{B}|}  
    \le 
    3 q^{7/8} \sqrt{|\mathcal{A}| |\mathcal{B}|}  \,.
\] 
    This completes the proof. 
$\hfill\Box$
\end{proof}

\begin{remark}
    We have used a direct way of  the proof of Theorem \ref{t:Fish-Vinh+} above, another approach is to notice that $\tilde{I}(a,b) = I(ga,gb)$, where the linear transformation $g$ is given by the formula $g(x,y) = (x+y,x-y)$. After that one can apply Theorem \ref{t:Fish-Vinh} with the sets $g^{-1}(\mathcal{A})$, $g^{-1}(\mathcal{B})$. 
\end{remark}


\section{On an application of group actions} 
\label{sec:group_Fish}

In this section we discuss another approach to results of Fish--type, namely, we consider an intermediate situation between Theorems \ref{t:Fish_new} and \ref{t:Fish-Vinh}: our set $\mathcal{A} \subseteq \Z^2_q$ is an arbitrary but the set $\mathcal{B} \subseteq \Z^2_q$ is a Cartesian product. 
In this case one can deal with rather general $q$ (and not just squarefree). For simplicity, we do not do any regularization as in the previous section immediately assuming 
that all prime factors of $q$ are large. 

\bp 

In the proof we follow the methods from \cite{bourgain2008expansion} and \cite{sh_Kloosterman}.

\begin{theorem}
    Let $q$ be a positive odd integer,  
    and $\mathcal{A}, \mathcal{B} \subseteq \Z_q^2$ be sets, 
    $|\mathcal{A}| = \d q^2$, $\mathcal{B} = A\times B$, $|A| = \a q$, $|B| = \beta q$. 
    Suppose that all prime divisors of $q$ are at least $M$, 
    where 
\[
   M \ge C_1 \tau (q) \d^{-2} (\sqrt{\a \beta} (\log \log q)^{-1})^{-C_2} \,,
\]
    and 
    $C_1, C_2>0$ are absolute constants. 
    Then 
\begin{equation}\label{f:Fish_action}
    \Z^*_q \subseteq 
    \{ (a_1-b_1) (a_2 - b_2) ~:~ (a_1,a_2) \in \mathcal{A},\, (b_1,b_2) \in \mathcal{B} \} 
    \,.
\end{equation} 
\label{t:Fish_action}
\end{theorem} 
\begin{proof}
    Let $q=p^{\rho_1}_1 \dots p^{\rho_t}_t$, 
    where $p_j$ are different odd primes and $\rho_j$ are positive integers. 
    By our assumption $p_j \ge M$ for all $j\in [t]$.
    Without loosing of the generality one can take $\la = -1$ in  
    formula 
    \eqref{f:N=(A-A)(B-B)}.
    Recall that $\SL_2(\Z_q)$ acts on 
    $\Z_q$ 
    via M\"obius transformations: $x\to gx = \frac{ax+b}{cx+d}$, where $g=\left(\begin{array}{cc}
a & b \\
c & d
\end{array}\right)$ (for composite $q$ the equivalence is taken over $\Z^*_q$, of course).
    Since $\mathcal{B} = A\times B$ we can rewrite our equation \eqref{f:N=(A-A)(B-B)}  as 
\begin{equation}\label{f:ga=b}
    a=gb \,, \quad \quad a\in A\,, b\in B\,, g\in G \,,
\end{equation}
    where $G \subset \SL_2 (\Z_q)$ is the set of matrices of the form 
\[
g=
\left( {\begin{array}{cc}
	-\a & \a \beta+1 \\
	-1 & \beta \\
	\end{array} } \right)  \,, \quad \quad (\a,\beta) \in \mathcal{A} 
\,,
\]
    see \cite[Section 5]{sh_Kloosterman} or just make a direct calculation. 
    Clearly, $|G| = |\mathcal{A}|$. 
    Further by \cite[Lemma 15]{sh_Kloosterman} the {\it multiplicative energy} $\E(G)$ of the set $G$, that is, 
\[
    \E(G) = |\{ (g_1,g_2,g_3,g_4) \in G\times G \times G \times G ~:~ g_1 g_2^{-1} = g_3 g_4^{-1} \}|  
\]
    coincides with the number of the solutions  to the system 
\[
    \beta_1 - \beta_2 = \beta_3 - \beta_4 := s \,, \quad \quad 
    s(\a_1 - \a_3) = s(\a_2 - \a_4) = 0 \,, \quad \quad 
    \a_1 - \a_2 - \a_1 \a_2 s = \a_3 - \a_4 - \a_3 \a_4 s \,,
\]
    where $(\a_i,\beta_i) \in \mathcal{A}$, $i\in [4]$. 
    Let $s = d s'$, where $d$ is a divisor of $q$ and $s'$ is coprime to $q$. Taking $(\a_1,\beta_1), (\a_4,\beta_4) \in \mathcal{A}$, we find $\beta_2$, $\beta_3$ from the first equation  and $\a_2,\a_3$ modulo $q/d$ from the second one.
    Also, using $\a_3$ we can reconstruct $\a_2$ from the third equation, provided $d>1$.  
    In other words, for fixed $d$ there are $q/d$ possibilities for $s'$ and $d$ possibilities for $\a_3$. 
    Finally, if $d=1$, then we have at most $q|G|^2$ solutions. 
    Thus we 
    obtain  
    the bound  
\begin{equation}\label{tmp:12.01_1}
    \E(G) \le 
    |G|^2 \sum_{d|q}  \frac{q}{d} \cdot d  
    \le 
    \tau(q) q |G|^2 \,.
\end{equation}

    Now let us say a few words about representations of the group $\SL_2 (\Z_q)$, see  \cite[Sections 7, 8]{bourgain2008expansion}. 
    First of all, for any irreducible representation $\rho_q$ of $\SL_2 (\Z_q)$ we have $\rho = \rho_q = \rho_{p^{\rho_1}_1} \otimes \dots \otimes \rho_{p^{\rho_t}_t}$ and hence it is sufficient to understand the representation theory for $\SL_2 (\Z_{p^n})$, where $p$ is a prime number and $n$ is  a positive integer.  
    Now by \cite[Lemma 7.1]{bourgain2008expansion} we know that for any odd prime the dimension of any faithful irreducible representation of $\SL_2 (\Z_{p^n})$ is at least $2^{-1} p^{n-2} (p-1)(p+1)$.
    For an arbitrary  $r\le n$ we can consider the natural projection $\pi_r : \SL_2 (\Z_{p^n}) \to \SL_2 (\Z_{p^r})$ and let $H_r = \mathrm{Ker\,} \pi_r$. One can show that the set $\{ H_r \}_{r\le n}$ gives all normal subgroups of $\SL_2 (\Z_{p^n})$ and hence any nonfaithful irreducible representation arises as a faithful irreducible representation of $\SL_2 (\Z_{p^r})$ for a certain $r<n$.  Anyway, we see that  the multiplicity (dimension) $d_\rho$ of any nontrivial irreducible representation $\rho$ of  $\SL_2 (\Z_{p^n})$ is at least $p/3 \ge M/3$. 

    Applying  estimate \eqref{tmp:12.01_1}, using the formula for $\E(G)$ via the representations and taking into account the obtained lower bound for the multiplicities of the representations, we get
\[
    \frac{M\| \FF{G}\|^4_{op}}{3|\SL_2 (\Z_q)|}  \le \frac{1}{|\SL_2 (\Z_q)|} \sum_{\rho} d_\rho \| \FF{G} (\rho) \FF{G}^* (\rho) \|_{}^{2} 
    = \E(G) \le \tau(q) q |G|^2  \,,
\]
    and hence 
\begin{equation}\label{tmp:12.01_2}
    \| \FF{G}\|_{op} \le |G| \cdot \left( \frac{3\tau(q)}{M\d^2} \right)^{1/4} := \frac{|G|}{K} \,,
\end{equation}
    where by $\| \FF{G}\|_{op}$ we have denoted the maximum 
    of 
    the operator norm of matrices $\FF{G} (\rho)$  for all nontrivial representations $\rho$ and $\| \cdot \|$ is the usual Hilbert--Schmidt norm. 
    Thanks to our choice of $M$ one can see that bound \eqref{tmp:12.01_2} is nontrivial, that is, $K>1$. 
    Returning to \eqref{f:ga=b} and using the standard scheme, e.g., see \cite[Lemma 13, Section 5 and Section 6]{sh_Kloosterman}, we obtain 
\[
    N_{\mathcal{A}, \mathcal{B}} (\la) - \frac{|A||B||G| \_phi(q)}{q^2}
    \le 
    \sqrt{|A||B|} |G| q^{-1/k} \,,
\]
    where $k \sim \log q/\log K$. 
    Hence $N_{\mathcal{A}, \mathcal{B}} (\la) > 0$, provided $K \gg (\sqrt{\a \beta} (\log \log q)^{-1})^{-O(1)}$. 
    The last condition is equivalent to $M\gg \tau (q) \d^{-2} (\sqrt{\a \beta} (\log \log q)^{-1})^{-O(1)}$. 
    This completes the proof. 
$\hfill\Box$
\end{proof}

\section{On the covering numbers of difference sets}
\label{sec:covering}

Let us recall the definition of the covering number of a set, e.g., see \cite{bollobas2011covering}. 

\begin{definition}
    Let $\Gr$ be a finite abelian group with the group operation $+$, and let $A\subseteq \Gr$ be a set.
    We write 
\[
    \cov^{+}(A) = \cov (A) = \min \{ |X| ~:~ X\subseteq \Gr,\, A+X = \Gr \} 
\]
    and the quantity $\cov^{+} (A)$ is called the (additive) {\bf covering number} of $A$. 
\end{definition}

Having a finite ring $\mathcal{R}$ with two operations $+, \times$ we underline which covering number we use, writing $\cov^{+}$ or $\cov^\times$. 
It is known \cite[Corollary 3.2]{bollobas2011covering} that for any set $A \subseteq \Gr$ one has $\cov^{+} (A) = O\left(\frac{|\Gr|}{|A|} \log |A| \right)$ and 
the last 
bound is tight. 
In this section we study difference sets $A-A$, $A \subseteq \Z_q$ and show that $\cov^{\times} (A-A)$ is always small. 
First of all, let us make a remark about a connection between $\cov^{+}$ and $\cov^{\times}$ in a ring $\mathcal{R}$.

\begin{proposition}
    Let $\mathcal{R}$ be a finite ring, and $S \subseteq \mathcal{R}$ be a set.
    Then 
\begin{equation}\label{f:cov+/*}   
    \cov^\times (S-S) \le \cov^{+} (S) \,,
\end{equation}  
    provided all numbers $1,\dots, \cov^{+} (S)$ belong to $\mathcal{R}^*$. 
\label{p:cov+/*}   
\end{proposition}
\begin{proof}
    Let $S+X = \Z_q$ and $|X| = \cov^{+} (S) := k$. 
    For any $g\in \Z_q$ consider  $jg$, where $j=0,1,\dots, k$.
    By the pigeonhole principle there are different $j_1 \neq j_2$ such that $j_1 g \in S+x$ and $j_2 g \in S+x$ with the same $x\in X$. 
    It implies that $(j_1-j_2)g \in S-S$ and hence  $g\in (j_1-j_2)^{-1} (S-S)$, provided $(j_1-j_2)^{-1} \in \mathcal{R}^*$. 
    It remains to notice that $[-k,k]^{-1} \cdot (S-S) = [k]^{-1} \cdot (S-S)$.
    This completes the proof. 
$\hfill\Box$
\end{proof}

\bp

By the well--known consequence of the Ruzsa covering lemma \cite[Section 2.4]{TV}, we have  for any finite group $\Gr$ and a set  $A\subseteq \Gr$ that for a certain set $Z\subseteq \Gr$ one has 
\begin{equation}\label{f:Ruzsa_cov}  
    \Gr \subseteq A-A + Z \,, \quad \quad |Z| \le |\Gr|/|A| \,.
\end{equation}
In particular, it means that $\cov^{+} (A-A) \le |\Gr|/|A|$. 
Thus Proposition \ref{p:cov+/*} gives us

\begin{corollary}
    Let $\mathcal{R}$ be a finite ring,  $A \subseteq \mathcal{R}$ be a set, and  $|A| = \a |\mathcal{R}|$.
    Then 
\[
    \cov^\times (2A-2A) \le \a^{-1} \,,
\]
    provided all numbers $1,\dots, [\a^{-1}]$ belong to $\mathcal{R}^*$.
\label{c:2A-2A_cov}
\end{corollary}

Using 
the same 
method, one can 
estimate 
the multiplicative covering number of a Bohr set in $\Z_p$ ($p$ is a prime number)
\[
    \mathcal{B} (\Gamma, \eps) = \{ x\in \Z_p ~:~ \| x\gamma/p\| \le \eps,\, \forall \gamma \in \Gamma \} 
    \quad \eps \in (0,1],\, \quad \Gamma \subseteq \Z_p \,,
\]
namely, we have 
\[\cov^{\times} (\mathcal{B} (\Gamma, \eps)) \le \eps^{-|\Gamma|} \,.
\]

\bp 

It is interesting to decrease the number of summands in Corollary \ref{c:2A-2A_cov}. 
To this end 
let 
us obtain the main result of this section. 

\begin{theorem}
    Let $q$ be a positive integer, $A \subseteq \Z_q$ be a set, $|A| = \a q$.
    Suppose that the least prime factor of $q$ 
     greater than 
    $2 \a^{-1}+3$. 
    Then 
\begin{equation}\label{f:covm_A-A}  
    \cov^\times (A-A) \le \frac{1}{\a} + 1 \,.
\end{equation}
    More concretely,  $[k_*]^{-1} \cdot (A-A) = \Z_q$ for a certain $k_* \le \a^{-1}+1$. 
\label{t:covm_A-A}   
\end{theorem}
\begin{proof}
    Let $p_1$ be the least prime factor of $q$. 
    By our assumption we know that $p_1 \ge 2 \a^{-1}+3$. 
    Write $p_1=2k+1$ and take $\Lambda = \{0,1,\dots, k_*\}$, where $\lceil  \a^{-1} - 1 \rceil + 1 = k_* \le k$. 
    Then one has  $Y:= (\Lambda - \Lambda)\setminus \{0\} \subseteq \Z^*_q$.
    First of all, consider $n\in \Z_q^*$ and form the set $n\cdot \Lambda + A$.
    Since $|\Lambda||A| = (k_* +1) \a q > q$, it follows that there are different $\la_1,\la_2 \in \Lambda$ such that 
\[
    n\la_1 + a_1 \equiv  n\la_2  + a_2 \pmod q \,,
\]
    where $a_1, a_2 \in A$ and $a_1 \neq a_2$. 
    Hence $n\in Y^{-1} (A-A)$ and thus $\Z_q^* \subseteq Y^{-1} (A-A)$. 
    Also, notice that as in Proposition \ref{p:cov+/*}  one has   $Y^{-1} (A-A) = [k_*]^{-1} \cdot (A-A)$. 

    Now let $n=n' q_1$, where $q_1|q$ and $n'$ is coprime to $q$. 
    By the pigeonhole principle there is $B \subseteq \Z_{q/q_1}$ and $s\in \Z_q$ such that $q_1 B+s \subseteq A$ and the density of $B$ in $\Z_{q/q_1}$ is at least $\a$.  In particular, we have $q_1 (B-B) \subseteq A-A$. 
    By the same argument as above one has $n' \equiv y^{-1} (b_1-b_2) \pmod {q/q_1}$, where $y\in Y$ and $b_1,b_2 \in B$.
    It follows that $n \equiv y^{-1} (a_1-a_2) \pmod q$ as required.  
    Thus we have proved that $[k_*]^{-1} (A-A) = \Z_q$
    and hence 
    $\cov^\times (A-A) \le k_* \le \a^{-1}+1$.
    This completes the proof. 
$\hfill\Box$
\end{proof}

\begin{remark}
    After the paper was written the author was informed by A. Fish that Theorem \ref{t:covm_A-A} holds in greater generality, namely, for any measure preserving system the same is true for the set of return times of a set of positive measure. 
\label{r:Fish}
\end{remark} 

Theorem \ref{t:covm_A-A} implies a consequence about the multiplicative covering numbers of the intersections of difference sets  
in the spirit of paper \cite{{Stewart1983}}, see \cite[Theorems 1,3]{Stewart1983}.

\begin{corollary}
    Let $q$ be a positive integer, and $A_1,\dots, A_k \subseteq \Z_q$ be sets, $|A_i| = \a_i q$, $i\in [k]$. 
    Suppose that the least prime factor of $q$  
     greater than 
    $2 (\a_1 \dots \a_k)^{-1}+3$.
    Then 
\begin{equation}\label{f:A_i-A_i}
    \cov^{\times} \left(\bigcap_{i=1}^k (A_i - A_i) \right) \le \frac{1}{\a_1 \dots \a_k} + 1 \,. 
\end{equation}
\label{c:A_i-A_i}
\end{corollary}
\begin{proof}
    Put $A_{\vec{s}} = A_1 \cap (A_2-s_1) \cap \dots (A_k-s_{k-1})$, where $\vec{s} = (s_1,\dots, s_{k-1}) \in \Z^{k-1}_q$. 
    We have $\sum_{\vec{s}} |A_{\vec{s}}| = |A_1| \dots |A_k|$ and hence there is $\vec{s}_*$ such that $|A_{\vec{s}_*}| \ge \a_1 \dots \a_k q$. 
    Clearly, for any $\vec{s}$ one has 
\[
    A_{\vec{s}} - A_{\vec{s}} \subseteq \bigcap_{i=1}^k (A_i - A_i) \,.
\]
    Applying Theorem \ref{t:covm_A-A} with $A = A_{\vec{s}_*}$,  we obtain bound \eqref{f:A_i-A_i}. 
    This completes the proof. 
$\hfill\Box$
\end{proof}

\bp

As we have seen before 
Corollary \ref{c:2A-2A_cov}  and Theorem \ref{t:covm_A-A} give us some bounds for the multiplicative covering numbers of difference sets. 
On the other hand, one can see that Theorem \ref{t:covm_A-A} does not hold for, say, nonzero  shifts of Bohr sets, for the sumsets $A+A$, for the higher sumsets $nA$, $n>2$ and so on. 
Indeed, consider the following 

\bp 

{\bf Example.}
{\it 
Let $p$ be a prime number and $S = [p/3,2p/3)$ or $S=\pm [p/6,p/3)$ to make $S$ symmetric. 
Then the equation $a+b \equiv c \pmod p$ has no solutions in $a,b,c\in S$. 
Further, we have $|S| \gg p$ but it is easy to see that $\cov^\times (S)$ is unbounded. 
Indeed, if $SX = \Z_p$ for a set $X$ with $|X| = O(1)$, then we obtain a coloring of $\Z_p$ with a finite number of colors and every color has no solutions to our equation $a+b \equiv c \pmod p$. It gives us a contradiction with the famous  Schur theorem, see  \cite{schur1916kongruenz} (actually, it implies  $\cov^\times (S) \gg \log p/\log \log p$). 

In particular, we see that $\cov^\times (X+s)$ can be much larger than $\cov^\times (X)$ for a set $X$ and a nonzero $s$. 
}

\bp 

Proposition \ref{p:cov+/*} implies that any syndetic
(i.e. having bounded gaps between its consecutive elements)  
set $S\subseteq \F_p$, 
$|S| \gg p$ has $\cov^{\times} (S-S) = O(1)$. On the other hand, thanks to inclusion \eqref{f:Ruzsa_cov}  any set of the form $A-A$, where  $A\subseteq \F_p$, $|A| \gg p$ is  syndetic (with the gap depending on $A$ but not just on $p/|A|$, of course). Thus it is natural  to ask about 
a 
generalization of Theorem \ref{t:covm_A-A} to the family of  syndetic sets. Nevertheless, taking  $S=\{ 1+ kM\}_{k\in [(p-1)/M]}$, $M\ge 4$ and $p \equiv 2 \pmod M$, say, we see that $S$ is a syndetic set and $S$ has no solutions to the equation $a+b \equiv c \pmod p$. 
Thus as in the example above we see that $\cov^\times (S)$ is unbounded. 

\begin{remark}
    A dual form of Theorem \ref{t:covm_A-A} has no place, namely, there is a set $A\subseteq \Z_p$, $|A| \gg p$ such that $\cov^{+} (A/A) \gg \log p$.  
    In other words,  $\cov^{+} (A/A)$ is close to the maximal possible value. 
    To see this just take $A$ to be the set of all quadratic residues, e.g., see \cite[Proposition 14]{SS_higher}. 
\end{remark} 


Finally, let us give another proof of a variant of Theorem  \ref{t:Fish_intr} via our covering 
Theorem  \ref{t:covm_A-A}. 
Notice that the number $d$ below can be a non--divisor of $q$.

\begin{theorem}
    Let $q$ be a 
    positive integer, 
    $A, B\subset \Z_q$ be sets, $|A| = \a q$, $|B|=\beta q$, and let us assume  that $\a \ge \beta$. 
    Suppose that the least prime factor of $q$ 
     greater than 
    $2 \beta^{-1}+3$. 
    Then there is $d\neq 0$ 
    with  
\begin{equation}\label{f:Fish_p1_1}
    d \le \a^{-\beta^{-1} -1} \,,
\end{equation}
    and such that 
\begin{equation}\label{f:Fish_p1_2}
    d  \cdot \Z_q \subseteq (A-A)(B-B) \,.
\end{equation}
\label{t:Fish_p1}
\end{theorem} 
\begin{proof} 
    Applying   Theorem  \ref{t:covm_A-A} with $A=B$, we find a set $X\subseteq \Z_q$, $n:=|X| \le \beta^{-1} +1$ such that $X(B-B) = \Z_q$. 
    Let $X = \{ x_1,\dots,x_n\}$ and $\vec{x} = (x_1,\dots,x_n) \in \Z^n_q$. 
    Considering the collection of the sets $A^n + j \cdot \vec{x} \subseteq \Z^n_q$, $j\ge 1$, we see that there is $0<d \le \a^{-n}$ with $d\cdot X \subseteq A-A$. 
    Hence 
\[
    (A-A)(B-B) \supseteq d\cdot X (B-B) \supseteq d \cdot \Z_q 
\]
    as required. 
    It remains to notice that 
\[
    d \le \a^{-n} \le \a^{-\beta^{-1} -1} \,.
\]
    This completes the proof. 
$\hfill\Box$
\end{proof}

\section{Concluding remarks}
\label{sec:concluding}

Let us discuss other approaches to Theorem \ref{t:Fish_intr}.
First of all, recall the well--known Furstenberg's result \cite{furstenberg2014recurrence}. 

\begin{theorem}
    Let $n$ be a positive integer, $\d \in (0,1]$ be a real number, and $S$ be a set of size $n$. 
    Then for all sufficiently large $N \ge N(\d,n)$ an arbitrary set 
    $A\subseteq [N]\times [N]$, $|A| \ge \delta N^2$ contains the set $\a+\beta\cdot S$ for some $\a$ and $\beta \neq 0$. 
\label{t:Sz_mult}
\end{theorem}

Quantitative bounds for 
$N(\d,n)$
from  
Theorem \ref{t:Sz_mult} can be found in \cite{shelah1988primitive}.

\begin{corollary}
    Let $q$ be a prime number, $\mathcal{A} \subseteq \F^2_q$, $|\mathcal{A}| = \d q^2$ and $A,B \subseteq \F_q$, $|A| = \a_* q$, $|B| = \beta_* q$. 
    Then there is a decreasing positive function $\_phi$ such that 
    if $\min \{\a_*,\beta_*,\d\} \ge \_phi (q)$, then formula \eqref{f:F=(A-A)(B-B)} takes place for $\mathcal{B} = A\times B$, any $\la \in \Z^*_q$ and $d=1$. 
\label{c:Fish_density}
\end{corollary}
\begin{proof}
    Take $S=S_1=[k]\times [k]$ or $S=S_2 = \{ (2j,2j) ~:~ j\in [k]\}$ for a certain positive integer $k$. 
    Applying Theorem \ref{t:Sz_mult} with $n=|S|$ and $A=\mathcal{A}$, we see that  for some $\a$, $\beta \neq 0$ the following holds  $\a+\beta\cdot S \subseteq \mathcal{A}$ and hence to solve \eqref{f:F=(A-A)(B-B)} with $d=1$ it is sufficiently to find for any $\la \in \Z^*_q$ some elements $a\in \beta^{-1} (A-\a)$, $b\in \beta^{-1} (B-\a)$ and $(t_1,t_2) \in S$ such that 
\[
    (t_1 - a) (t_2-b) \equiv \la \pmod q \,. 
\]
    If for a certain absolute constant $C>0$ one has $k \gg \min^{-C} \{\a_* ,\beta_*,\d\}$, then
    for $S=S_2$ the last equation has a solution thanks to the famous Bourgain--Gamburd machine \cite{bourgain2008uniform} (see details in \cite{s_BG}, say) and for $S=S_1$ (actually, for any dense subset of $S_1$) the latter fact was 
    obtained 
    in \cite[Theorem 3]{s_BG}.  
    This completes the proof. 
$\hfill\Box$
\end{proof}

\bp 

    The author does not know  how to obtain Corollary \ref{c:Fish_density} for composite $q$ because there is no control over divisors of $\beta$ in Theorem \ref{t:Sz_mult}.  It would be interesting to say something about prime factors of the dilation $\beta$. 

\bp 

We finish this section with a problem (it is interesting in its own right from a combinatorial point of view), which potentially gives another proof of Corollary \ref{c:Fish_density} thanks to  \cite[Theorem 3]{s_BG}.

\bp 

{\bf Problem.} 
{\it
 Let $n$ be a positive integer, $\d,\kappa \in (0,1]$ be real numbers.
    Then for all sufficiently large $N\ge N(\d,\kappa,n)$  an arbitrary set 
    $A\subseteq [N]\times [N]$, $|A| \ge \delta N^2$ contains the set $\a+\beta\cdot S$ for some $\a$ and $\beta$, where $S \subseteq [n]\times [n]$ is any set of  size $n^{1+\kappa}$.     
}

\bp 

Of course, some estimates on $N(\d,\kappa,n)$ follow from Theorem \ref{t:Sz_mult} but maybe it is possible to obtain a better bound. 


\bibliographystyle{abbrv}

\bibliography{main}{}

\begin{thebibliography}{10}

\bibitem{bollobas2011covering}
B.~Bollob{\'a}s, S.~Janson, and O.~Riordan.
\newblock On covering by translates of a set.
\newblock {\em Random Structures \& Algorithms}, 38(1-2):33--67, 2011.

\bibitem{bourgain2008sum}
J.~Bourgain.
\newblock {The sum-product theorem in $\mathbb{Z}_q$ with $q$ arbitrary}.
\newblock {\em Journal d'Analyse Math{\'e}matique}, 106(1):1--93, 2008.

\bibitem{bourgain2008expansion}
J.~Bourgain and A.~Gamburd.
\newblock Expansion and random walks in $\mathrm{SL}_d (\mathbb{Z}/p^n
  \mathbb{Z})$: I.
\newblock {\em Journal of the European Mathematical Society}, 10(4):987--1011,
  2008.

\bibitem{bourgain2008uniform}
J.~Bourgain and A.~Gamburd.
\newblock Uniform expansion bounds for cayley graphs of $\mathrm{SL}_2
  (\mathbb{F}_p)$.
\newblock {\em Annals of Mathematics}, pages 625--642, 2008.

\bibitem{fish2018product}
A.~Fish.
\newblock On product of difference sets for sets of positive density.
\newblock {\em Proceedings of the American Mathematical Society},
  146(8):3449--3453, 2018.

\bibitem{furstenberg2014recurrence}
H.~Furstenberg.
\newblock {\em Recurrence in ergodic theory and combinatorial number theory},
  volume~10.
\newblock Princeton University Press, 2014.

\bibitem{gyarmati2008equationsI}
K.~Gyarmati and A.~S{\'a}rk{\"o}zy.
\newblock {Equations in finite fields with restricted solution sets. I
  (Character sums)}.
\newblock {\em Acta Mathematica Hungarica}, 118(1-2):129--148, 2008.

\bibitem{gyarmati2008equationsII}
K.~Gyarmati and A.~S{\'a}rk{\"o}zy.
\newblock {Equations in finite fields with restricted solution sets. II
  (Algebraic equations)}.
\newblock {\em Acta Mathematica Hungarica}, 119(3):259--280, 2008.

\bibitem{hart2007sum}
D.~Hart, A.~Iosevich, and J.~Solymosi.
\newblock {Sum-product estimates in finite fields via Kloosterman sums}.
\newblock {\em International Mathematics Research Notices},
  2007(9):rnm007--rnm007, 2007.

\bibitem{pach2013ramsey}
P.~P. Pach.
\newblock Ramsey type results on the solvability of certain equation in
  $\mathbb{Z}_m$.
\newblock {\em Annual Volume 2013}, 13:41, 2013.

\bibitem{sarkozy2005sums}
A.~S{\'a}rk{\"o}zy.
\newblock On sums and products of residues modulo $p$.
\newblock {\em Acta Arithmetica}, 4(118):403--409, 2005.

\bibitem{SS_higher}
T.~Schoen and I.~D. Shkredov.
\newblock Higher moments of convolutions.
\newblock {\em J. Number Theory}, 133(5):1693--1737, 2013.

\bibitem{schur1916kongruenz}
I.~Schur.
\newblock {\"U}ber die kongruenz $x^m+ y^m \equiv z^m \pmod p$, jber.
\newblock {\em Deutsch. Math. Verein}, 25:114--116, 1916.

\bibitem{shelah1988primitive}
S.~Shelah.
\newblock Primitive recursive bounds for van der waerden numbers.
\newblock {\em Journal of the American Mathematical Society}, 1(3):683--697,
  1988.

\bibitem{sh_mono}
I.~D. Shkredov.
\newblock {On monochromatic solutions of some nonlinear equations in
  $\mathbb{Z}/p\mathbb{Z}$}.
\newblock {\em Mat. Zametki}, 88(4):625--634, 2010.

\bibitem{sh_Kloosterman}
I.~D. Shkredov.
\newblock {Modular hyperbolas and bilinear forms of Kloosterman sums}.
\newblock {\em Journal of Number Theory}, 220:182--211, 2021.

\bibitem{s_BG}
I.~D. Shkredov.
\newblock {On a girth–free variant of the Bourgain--Gamburd machine}.
\newblock {\em arXiv:2111.05751v1}, 2022.

\bibitem{shparlinski2008solvability}
I.~E. Shparlinski.
\newblock On the solvability of bilinear equations in finite fields.
\newblock {\em Glasgow Mathematical Journal}, 50(3):523--529, 2008.

\bibitem{Stewart1983}
C.~L. Stewart and R.~Tijdeman.
\newblock {\em On density-difference sets of sets of integers}, pages 701--710.
\newblock Birkh{\"a}user Basel, Basel, 1983.

\bibitem{TV}
T.~Tao and V.~Vu.
\newblock {\em Additive combinatorics}, volume 105 of {\em Cambridge Studies in
  Advanced Mathematics}.
\newblock Cambridge University Press, Cambridge, 2006.

\bibitem{vinh2011szemeredi}
L.~A. Vinh.
\newblock {The Szemer\'{e}di--Trotter type theorem and the sum-product estimate
  in finite fields}.
\newblock {\em European Journal of Combinatorics}, 32(8):1177--1181, 2011.

\end{thebibliography}


\begin{thebibliography}{99}



\bibitem{BG}
{\sc J. Bourgain, A. Gamburd, } 
{\em Uniform expansion bounds for Cayley graphs of $SL_2(\F_p)$, } 
Ann. of Math., 167(2):625--642, 2008.


\bibitem{Fish}
{\sc A. Fish, }
{\em On product of difference sets for sets of positive density, } Proceedings of the American Mathematical Society 146, no. 8 (2018): 3449--3453.


\bibitem{s_BG}
{\sc I.D. Shkredov, }
{\em On a girth–free variant of the Bourgain–Gamburd machine, } 
arXiv:2111.05751v1 [math.NT] 10 Nov 2021.


\bibitem{Vinh} {\sc L. A. Vinh,}
{\em The Szemer\'edi-Trotter type theorem and the sum-product estimate in finite fields,}  European J. Combin. {\bf 32}(8): 1177--1181, 2011.


\end{thebibliography}

\end{document}